\documentclass[12pt]{amsart}
\usepackage{amscd}
\usepackage{amsmath}
\usepackage{verbatim}

\usepackage[cmtip,arrow]{xy}

\usepackage{pb-diagram,pb-xy}

\usepackage{amssymb}

\newtheorem{thm}[equation]{Theorem}
\newtheorem{prop}[equation]{Proposition}

\newtheorem{lemma}[equation]{Lemma}
\newtheorem{cor}[equation]{Corollary}

\theoremstyle{definition}
\newtheorem{rem}[equation]{Remark}

\newcommand{\CH}{\mathop{\mathrm{CH}}\nolimits}

\newcommand{\OO}{\operatorname{\mathrm{O}}}

\newcommand{\Ch}{\mathop{\mathrm{Ch}}\nolimits}

\newcommand{\cdim}{\mathop{\mathrm{cdim}}\nolimits}

\newcommand{\mult}{\operatorname{mult}}

\newcommand{\Z}{\mathbb{Z}}
\newcommand{\F}{\mathbb{F}}

\newcommand{\Spec}{\operatorname{Spec}}

\newcommand{\pt}{\mathbf{pt}}

\newcommand{\disc}{\operatorname{disc}}

\newcommand{\compose}{\circ}

\renewcommand{\phi}{\varphi}

\usepackage[hypertex]{hyperref}

\title
{Incompressibility of orthogonal grassmannians}


\author
{Nikita A. Karpenko}

\address
{
UPMC Sorbonne Universit\'es\\
Institut de Math\'ematiques de Jussieu\\
F-75252 Paris\\
FRANCE}

\address
{{\it Web page:}
{\tt www.math.jussieu.fr/\~{ }karpenko}}

\email {karpenko {\it at} math.jussieu.fr}

\date
{July 2011. Revised: October 2011.}

\begin{document}

\begin{abstract}
We prove the following conjecture due to Bryant Mathews (2008).
Let $Q$ be the orthogonal grassmannian of totally isotropic $i$-planes
of a non-degenerate quadratic form $q$ over an arbitrary field (where $i$ is an integer
satisfying $1\leq i\leq(\dim q)/2$).
If the degree of each closed point on $Q$ is divisible by $2^i$ and the
Witt index of $q$ over the function field of $Q$ is equal to $i$, then the variety $Q$ is
$2$-incompressible.

\bigskip
\noindent
{\bf Incompressibilit\'e de grassmanniennes orthogonales.}
Nous d\'{e}montrons la conjecture ci-dessous due \`{a} Bryant Mathews (2008).
Soit $Q$ la grassmannienne orthogonale des $i$-plans totalement isotropes
d'une forme quadratique non-d\'{e}g\'{e}n\'{e}r\'{e}e $q$  sur un corps arbitraire (o\`{u}
$i$ est un entier satisfaisant $1\leq i\leq(\dim q)/2$).
Si le degr\'{e} de tout point ferm\'{e} sur $Q$ est divisible par $2^i$ et
l'indice de Witt de la forme $q$ au-dessus du corps des fonctions de $Q$
est \'{e}gal \`{a} $i$, alors la vari\'{e}t\'{e} $Q$ est $2$-incompressible.
\end{abstract}

\maketitle


Theorem \ref{main}, proved below, has been conjectured in the Ph.D. thesis
\cite[page 24]{MR2713320}
(a preprint with the conjecture appeared one year earlier).

We start with some development of the theory of canonical dimension of
general projective homogeneous varieties (which might be of independent interest).
We fix a prime $p$.
Let $G$ be a semisimple affine algebraic group over a field $F$
such that $G_E$ is of inner type for some finite galois field extension
$E/F$ of degree a power of $p$ ($E=F$ is allowed).
Let $X$ be a projective $G$-homogeneous $F$-variety.
We refer to \cite{canondim} for a definition and discussion of the notion
of {\em canonical $p$-dimension} $\cdim_pX$ of $X$.
Actually, canonical $p$-dimension is defined in the context of more general algebraic varieties.
For any irreducible smooth projective variety $X$, $\cdim_pX$ is the
minimal dimension of a closed subvariety $Y\subset X$ with
a $0$-cycle of $p$-coprime degree on $Y_{F(X)}$.
Recall that a smooth projective $X$ is {\em $p$-incompressible}, if it is
irreducible and $\cdim_pX=\dim X$.

\begin{prop}
\label{prop}
For $d:=\cdim_p X$, there exist a cycle class
$\alpha\in\CH^dX_{F(X)}$ (over $F(X)$) of codimension $d$ and a cycle class
$\beta\in\CH_dX$ (over $F$) of dimension $d$ such that the degree of the
product $\beta_{F(X)}\cdot\alpha$ is not divisible by $p$.
\end{prop}

\begin{proof}
We use Chow motives with coefficients in $\F_p:=\Z/p\Z$ (as defined in \cite[Chapter XII]{EKM}) and write
$\Ch$ for the Chow group $\CH$ modulo $p$.

Let $U(X)$ be the {\em upper} motive of $X$.
By definition, $U(X)$ is a direct summand of the motive $M(X)$ of $X$ such that $\Ch^0 U(X)\ne0$.
By \cite[Theorem 5.1 and Proposition 5.2]{canondim}, $U(X)$ is also a direct summand of $M(X)(d-m)$, where
$m:=\dim X$.
The composition
$$
M(X)\to U(X)\to M(X)(d-m)
$$
is given by a correspondence
$f\in\Ch^d(X\times X)$;
the composition
$$
M(X)(d-m)\to U(X)\to M(X)
$$
is given by a correspondence
$g\in\Ch_d(X\times X)$.
The composition of correspondences $g\compose f\in\Ch_m(X\times X)$ is a projector on $X$
such that $U(X)=(X,g\compose f)$.
In particular, the {\em multiplicity} $\mult(g\compose f)$ of the correspondence $g\compose
f$ is $1\in\F_p$.
Taking for $\alpha$ an integral representative of the pull-back of $f$ with respect to
the morphism
$$
\Spec F(X)\times X\to X\times X
$$ induced by the generic
point of the first factor, and taking for $\beta$
an integral representative of
the push-forward of $g$
with respect to the projection of $X\times X$ onto the first factor, we
get that
\begin{equation*}
\deg(\beta_{F(X)}\cdot\alpha)\pmod{p}=\mult(g\compose f)=1\in\F_p.
\qedhere
\end{equation*}
\end{proof}

\begin{cor}
\label{cor-prop}
The canonical $p$-dimension
$\cdim_p X$ of $X$ is the minimal integer $d$ such that there exists a cycle class
$\alpha\in\Ch^dX_{F(X)}$ and a cycle class
$\beta\in\Ch_dX$ with $\deg(\beta_{F(X)}\cdot\alpha)=1\in\F_p$.
\end{cor}

\begin{proof}
We only need to show that $\cdim_pX\leq d$.
The proof is similar to \cite[Proof of $\leq$ in Theorem 5.8]{MR2258262}.
Since $\deg(\beta_{F(X)}\cdot\alpha)=1\in\F_p$ for some $\beta\in\Ch_dX$ (and some $\alpha$),
there exists a closed irreducible $d$-dimensional subvariety $Y\subset X$
such that $\deg([Y]_{F(X)}\cdot\alpha)\ne0\in\F_p$ (with the same $\alpha$).
Since the product $[Y]_{F(X)}\cdot\alpha$ can be represented by a cycle on
$Y_{F(X)}$, the variety $Y_{F(X)}$ has a $0$-cycle of $p$-coprime degree.
Therefore $\cdim_pX\leq\dim Y=d$.
\end{proof}

\begin{cor}
\label{cor}
In the situation of Proposition \ref{prop},
for any field extension $L/F$, the change of field homomorphism
$\Ch_dX\to\Ch_dX_L$ is non-zero.
\end{cor}

\begin{proof}
The image of $\beta\in\Ch_dX$ in $\Ch_dX_L$ is non-zero because
$\deg(\beta_{L(X)}\cdot\alpha_{L(X)})=1$.
\end{proof}

\begin{rem}
If the variety $X$ is {\em generically split} (meaning that the motive of
$X_{F(X)}$ is a sum of Tate motives which, in particular, implies that the adjoint algebraic group
acting on $X$ is of
inner type), then \cite[Theorem 5.8]{MR2258262} says that $\cdim_pX$ is
the minimal $d$ with non-zero $\Ch_dX\to\Ch_dX_L$ for any $L$.
Corollary \ref{cor} can be considered as a generalization of a part of \cite[Theorem 5.8]{MR2258262}
to the case of a projective $G$-homogeneous variety $X$ which is not
necessarily generically split with $G$ not necessarily of inner type.
Note that the statement of \cite[Theorem 5.8]{MR2258262} in whole fails in
such generality.
Corollary \ref{cor-prop} is its correct replacement
(giving the original statement in the case of generically split $X$).
\end{rem}

\begin{lemma}
\label{lemma}
In the situation of Proposition \ref{prop}, let
$\alpha,\alpha'\in\Ch^dX_{F(X)}$ and
$\beta,\beta'\in\Ch_dX$ be cycle classes with
$\deg(\beta_{F(X)}\cdot\alpha)=1=\deg(\beta'_{F(X)}\cdot\alpha')$.
Then
$$
\deg(\beta_{F(X)}\cdot\alpha')\ne0\ne\deg(\beta'_{F(X)}\cdot\alpha).
$$
\end{lemma}

\begin{proof}
We fix an algebraically closed field containing $F(X)$ and write $\bar{\cdot}$
when considering a variety or a cycle class over that field.
The surjectivity of the pull-back with respect to the
flat morphism $\Spec F(X)\times X\to X\times X$ induced by the generic point of the first factor of the product $X\times
X$, tells us that the group $\Ch^d(\bar{X}\times\bar{X})$ contains a
{\em rational} (i.e., ``coming from $F$'') cycle class of the form
$[\bar{X}]\times\bar{\alpha}+\dots+\bar{\gamma}\times[\bar{X}]$
with some $\gamma\in\Ch^dX_{F(X)}$, where $\dots$ is in the sum of products
$\Ch^i\bar{X}\otimes\Ch^j\bar{X}$ with $0<i,j<d$ and $i+j=d$.
Multiplying by $[\bar{X}]\times\bar{\beta}$, we get a rational cycle class
of the form $[\bar{X}]\times[\pt]+\dots+\bar{\gamma}\times\bar{\beta}$,
where $\pt$ is a rational point on $\bar{X}$ and $\dots$ is now in the sum
of $\Ch^i\bar{X}\otimes\Ch_i\bar{X}$ with $0<i<d$.
The composition of the obtained correspondence with itself equals
$[\bar{X}]\times[\pt]+\dots+\deg(\gamma\cdot\beta)(\bar{\gamma}\times\bar{\beta})$.
Since an appropriate power of this correspondence is a multiplicity $1$ {\em
projector} (cf. \cite[Corollary 3.2]{hypernew-tignol} or \cite{upper}) and
$d=\cdim_pX$, it follows by \cite[Theorem 5.1]{canondim} that
$\deg(\gamma\cdot\beta)\ne0$.
Now multiplying $[\bar{X}]\times\bar{\alpha}+\dots+\bar{\gamma}\times[\bar{X}]$
by $\bar{\beta}\times[\bar{X}]$, transposing, and raising to ($p-1$)th power (by means of composition of
correspondences), we get a rational cycle of the form
$[\bar{X}]\times[\pt]+\dots+\bar{\alpha}\times\bar{\beta}$.

Similarly, there is a rational cycle of the form
$[\bar{X}]\times[\pt]+\dots+\bar{\alpha'}\times\bar{\beta'}$.
One of its compositions with the previous one produces
$[\bar{X}]\times[\pt]+\dots+\deg(\beta'\cdot\alpha)(\bar{\alpha'}\times\bar{\beta})$,
therefore $\deg(\beta'\cdot\alpha)\ne0\in\F_p$.
The other composition produces
$[\bar{X}]\times[\pt]+\dots+\deg(\beta\cdot\alpha')(\bar{\alpha}\times\bar{\beta'})$,
so that $\deg(\beta\cdot\alpha')\ne0$.
\end{proof}

We specify as follows the settings of Proposition \ref{prop}.
The prime $p$ is now $2$.
The algebraic group $G$ is now $\OO^+(q)$
(in notation of \cite[\S23]{MR1632779})
for a non-degenerate quadratic
form $q$ (one may take $E=F$ if $\dim q$ is odd or $\disc q$ is
trivial, otherwise
$E$ can be the quadratic field extension of $F$ given by the discriminant of $q$).
We set $n:=\dim q$.
For any integer $i$ with $0\leq i\leq n/2$, let $Q_i$ be the
variety of $i$-dimensional totally isotropic subspaces in $q$.
In particular, $Q_0=\Spec F$.
For $i$ with $0<i<n/2$, $Q_i$ is a projective $G$-homogeneous
variety.

\begin{cor}
\label{cor6}
If $\cdim_2 Q_i=\cdim_2 Q'_{i-1}=\dim Q'_{i-1}$ for some some $i$ with
$0<i<n/2$, where $Q'_{i-1}$ is the orthogonal grassmannian of totally isotropic ($i-1$) planes of
a ($n-2$)-dimensional quadratic form $q'$ over $F(Q_1)$ Witt-equivalent to $q_{F(Q_1)}$,
then $\deg\CH Q_i\ni2^{i-1}$.
\end{cor}

\begin{proof}
The statement being trivial for $i=1$, we may assume that $i\geq2$.

For $d:=\cdim_2Q_i$, using Proposition \ref{prop},
we choose some $\alpha\in\CH^dQ_{i\,F(Q_i)}$ and
$\beta\in\CH_dQ_i$ with odd $\deg(\beta_{F(Q_i)}\cdot\alpha)$.
Note that $\cdim_2Q_{i\,F(Q_1)}=\cdim_2Q'_{i-1}=d$.
We construct some special $\alpha'\in\CH^dQ_{i\,F(Q_1)(Q_i)}$ and
$\beta'\in\CH_dQ_{i\,F(Q_1)}$ with $\deg(\beta'_{F(Q_1)(Q_i)}\cdot\alpha')=1$ as follows.
Let $Q_{1\subset i}$ be the variety of $(1,i)$-flags of totally isotropic
subspaces in $q$ together with the evident projections $Q_{1\subset i}\to
Q_1,Q_i$.
We define $\beta'$ as the pull-back via $Q_{1\subset i\,F(Q_1)}\to Q_{1\,F(Q_1)}$ followed by the
push-forward via $Q_{1\subset i\,F(Q_1)}\to Q_{i\,F(Q_1)}$ of the rational point class $l_0$ on
$Q_{1\,F(Q_1)}$.
We define $\alpha'$ as the product $z_{i-1}\dots z_1$, where $z_j$ is the pull-back via
$Q_{1\subset i\,F(Q_1)(Q_i)}\to Q_{1\,F(Q_1)(Q_i)}$ followed by the
push-forward via $Q_{1\subset i\,F(Q_1)(Q_i)}\to Q_{i\,F(Q_1)(Q_i)}$ of the class $l_j$ of a $j$-dimensional projective
subspace on $Q_{1\,F(Q_1)(Q_i)}$.
The degree condition on $\alpha'$ and $\beta'$ is satisfied by
\cite[Statement 2.15]{Vishik-u-invariant}.
Fixing an algebraically closed field containing $F(Q_1)(Q_i)$,
we see by Lemma \ref{lemma} that the product
$\bar{\beta}\cdot\bar{\alpha'}$ is an odd degree $0$-cycle class on $\bar{Q_i}$.
Moreover, the class $\bar{\beta}$ is rational.
Since $2\bar{z_j}$ is rational for every $j<(n-2)/2$ (by the reason that $2l_j$ is rational),
the class $2^{i-1}\bar{\beta}\bar{\alpha}$
is also rational and it follows that $2^{i-1}\in\deg\CH Q_i$.
\end{proof}

We come to the main result of this note.
It is known for $i=1$ by \cite{MR1992016} (the proof is essentially contained already
in \cite{vishik-summands}; the characteristic $2$ case has been treated later on in \cite{EKM}).
The case of maximal $i$, i.e., of $i=[n/2]$, is also known and is
discussed in the beginning of the proof.
For $i=2$ and odd-dimensional $q$, it has been proved in \cite{MR2713320}
(the proof for $i=2$ given here is different; in particular, it does not make use of the motivic decompositions
of \cite{MR2264459} for products of projective homogeneous varieties).

\begin{thm}
\label{main}
Let $q$ be a non-degenerate quadratic form over a field $F$.
Let $i$ be an integer satisfying $1\leq i\leq (\dim q)/2$.
If the degree of every closed point on $Q_i$ is divisible by $2^i$ and
the Witt index of the quadratic form $q_{F(Q_i)}$ equals $i$, then the variety
$Q_i$ is {\em $2$-incompressible} (i.e., $\cdim_2 Q_i=\dim Q_i$).
\end{thm}

\begin{proof}
We set $n:=\dim q$.
Note that for $i=n/2$ (and even $n$) the condition on closed points on
$Q_{n/2}$ ensures that $\disc q$ is non-trivial.
In particular, $Q_{n/2}$ is irreducible.

In general, for even $n$, the variety $Q_{n/2}$ is isomorphic to the orthogonal grassmannian of
totally isotropic ($n/2-1$)-planes of $q^1$ considered as a variety over $F$, where $q^1$ is any
$1$-codimensional non-degenerate subform in $q_K$, and $K$ is the quadratic
\'etale $F$-algebra given by the discriminant of $q$.
Therefore the statement of Theorem \ref{main} for $i=n/2$ follows from the
statement for $i=(n-1)/2$.
By this reason, we do not consider the case of $i=n/2$ below.
In particular, $Q_i$ below is a projective $G$-homogeneous variety.

We induct on $n$.
There is nothing to prove for $n<3$.
Below we are assuming that $n\geq3$.

Over the field $F(Q_1)$, the motive of $Q_{i\,F(Q_1)}$ decomposes as
follows (cf. \cite{gog}, \cite{MR1758562} or \cite{MR2178658}):
$$
M(Q'_{i-1})\oplus M(Q'_{i})\big((\dim Q_i-\dim Q'_i)/2\big)\oplus M(Q'_{i-1})(\dim Q_i-\dim Q'_{i-1}),
$$
where, as in Corollary \ref{cor6}, $Q'_j$ is the variety of $j$-dimensional totally isotropic subspaces
of a quadratic form $q'$ over the field $F(Q_1)$ such that $q_{F(Q_1)}$
is isomorphic to the orthogonal sum of $q'$ and a hyperbolic plane.
Since $n':=\dim q'=n-2<n$, the variety $Q'_{i-1}$ is
$2$-incompressible by the induction hypothesis (more precisely, the
induction hypothesis is applied if $i\geq2$, for $i=1$ the statement if
trivial).
Indeed, since the extension $F(Q_1)/F$ is a tower of a purely transcendental extension followed by a quadratic one,
the degree of any closed point on $Q'_{i-1}$ is
divisible by $2^{i-1}$;
the Witt index of $q'_{F(Q_1)(Q'_{i-1})}$ is $i-1$, that is,
the Witt index of $q_{F(Q_1)(Q'_{i-1})}$ is $i$ because the field extension
$F(Q_1)(Q'_{i-1})(Q_i)/F(Q_i)$ is purely transcendental.

By \cite[Theorem 1.1]{outer} (cf. \cite{gog}), it follows that the motive of
$Q'_{i-1}$ decomposes in a direct sum of one copy of $U(Q'_{i-1})$, shifts of
$U(Q'_j)$ with various $j\geq i$, and (in the case of even $n$ and non-trivial $\disc q$)
shifts of $U(Q'_{j\,K})$ with $j\geq i-1$ (where $K/F$ is the quadratic field extension corresponding to $\disc q$).
The motive of $Q'_i$ decomposes in a direct sum of shifts of
$U(Q'_j)$
and (in the case of even $n$ and non-trivial $\disc q$)
shifts of $U(Q'_{j\,K})$
with various $j\geq i$.
Note that the condition on the Witt index of the form $q_{F(Q_i)}$ ensures
that for any $j\geq i$ the motive $U(Q'_{i-1})$ is not isomorphic to
$U(Q'_j)$ (and $U(Q'_{i-1})\not\simeq U(Q'_{j\,K})$ anyway).
Therefore the complete motivic decomposition of $Q_{i\,F(Q_1)}$ contains
one copy of $U(Q'_{i-1})$, one copy of $U(Q'_{i-1})(\dim Q_i-\dim
Q'_{i-1})$ and no other shifts of $U(Q'_{i-1})$.

The complete decomposition of $U(Q_i)_{F(Q_1)}$
contains the summand $U(Q'_{i-1})$.
If it also contains the second (shifted) copy of $U(Q'_{i-1})$, then
$\cdim_2Q_i=\dim Q_i$ by \cite[Theorem 5.1]{canondim}, and we are done.
Otherwise,
by \cite[Lemma 1.2 and Remark 1.3]{gog},
$\cdim_2Q_i=\cdim_2Q'_{i-1}=\dim Q'_{i-1}$,
and we get by Corollary \ref{cor6} that $Q_i$ has a closed point of degree not divisible by $2^i$.
\end{proof}


\def\cprime{$'$}

\end{document}